\documentclass[12pt,a4paper]{article}
\usepackage{amsmath, amsfonts, ifthen, latexsym, amssymb, amsthm, amscd}
\usepackage{amsrefs,mdframed, etoolbox}
\usepackage{multirow}

\usepackage[utf8]{inputenc}

\usepackage[margin=1in]{geometry}

\usepackage[shortlabels]{enumitem}
\usepackage{graphicx,xcolor}
\usepackage{url}
\usepackage{hyperref}
\hypersetup{colorlinks=true,citecolor=blue,filecolor=blue,linkcolor=blue,urlcolor=blue}
\usepackage[title]{appendix}

\newtheorem{theorem}{Theorem} % Single numbers
\newtheorem*{theorem*}{Theorem}
\newtheorem{Theorem}[theorem]{Theorem}
\newtheorem{Lemma}[theorem]{Lemma}
\newtheorem{lemma}[theorem]{Lemma}
\newtheorem*{lemma*}{Lemma}

\newtheorem{proposition}[theorem]{Proposition}

\newtheorem{corollary}[theorem]{Corollary}

\theoremstyle{definition}
\newtheorem{definition}[theorem]{Definition}
\newtheorem{Question}[theorem]{Question}

\newtheorem{remark}[theorem]{Remark}
\newtheorem{example}[theorem]{Example}

\makeindex

\newcommand{\E}{{{\mathbb E}}}
\renewcommand{\Pr}{{\mathbb P}}
\newcommand{\Z}{\mathbb Z}

\newcommand\RR{\mathbb R}

\newcommand\NN{\mathbb N}
\newcommand\ZZ{\mathbb Z}

\newcommand{\al}{{\alpha}}

\newcommand{\be}{{\beta}}

\newcommand{\om}{{\omega}}
\newcommand{\eps}{{\varepsilon}}
\newcommand{\de}{{\delta}}
\newcommand{\De}{{\Delta}}
\newcommand{\ga}{{\gamma}}

\renewcommand{\phi}{{\varphi}}

\provideboolean{theno_parts}
\setboolean{theno_parts}{true}
\renewcommand{\part}[1][(a)]{%
    \ifthenelse{\boolean{theno_parts}}{%
        \begin{enumerate}[#1]%
            \item
    }{%
        \item
    }
    \setboolean{theno_parts}{false}%
}
\newcommand{\trap}{%
    \ifthenelse{\boolean{theno_parts}}{}{%
        \end{enumerate}
    }
    \setboolean{theno_parts}{true}
}

\newcommand{\newop}[2]{%
    \expandafter\def\csname #1\endcsname{\operatorname{#2}}
}

\newop{In}{In}
\newop{Out}{Out}
\newop{indeg}{indeg}
\newop{outdeg}{outdeg}
\newop{ordiam}{depth}
\newop{depth}{depth}
\newop{im}{im}

\newop{output}{\tt output}
\newop{shift}{\tt shift}
\newop{gate}{\tt gate}

\newop{Id}{\tt Id}

\renewcommand{\cal}[1]{{\mathcal{#1}}}

\DeclareFontFamily{U}{mathx}{\hyphenchar\font45}
\DeclareFontShape{U}{mathx}{m}{n}{<-> mathx10}{}
\DeclareSymbolFont{mathx}{U}{mathx}{m}{n}
\DeclareMathAccent{\widebar}{0}{mathx}{"73}
\renewcommand{\bar}{\widebar}

\newop{maxindeg}{maxindeg}
\newop{IN}{In}
\newop{OUT}{Out}
\newop{Ind}{Ind}
\newop{Pre}{Pre}
\newop{Post}{Post}
\newop{Pow}{Pow}

\newop{Impact}{Res}
\newop{Res}{Res}

\newop{Hull}{Hull}
\newop{codiam}{codiam}
\newcommand{\suchthat}{\text{ such that }}

\definecolor{red}{rgb}{1,0,0}
\definecolor{blue}{rgb}{0,0,1}

\title{On directed analogues of expander and hyperfinite graph sequences}
\author{Endre Cs{\'o}ka\thanks{E.Cs. was supported by the ERC Synergy grant No. 810115}, {\L}ukasz Grabowski\thanks{\L.G. was supported by the ERC under the European Union's Horizon 2020 research and innovation programme (Grant agreement No. 805495)}}
\date{}

\begin{document}
\maketitle
\begin{abstract}

We introduce and study analogues of expander and hyperfinite graph 
sequences in the context of directed acyclic graphs, which we call 
``extender'' and ``hypershallow'' graph sequences, respectively. Our main result is a 
probabilistic construction of non-hypershallow graph sequences.

\end{abstract}

\noindent

\tableofcontents

\section{Introduction}
Hyperfinite and expander graph sequences are perhaps the two most fundamental concepts studied in the theory of sparse graph limits. Hyperfinite graph sequences were explicitly introduced in~\cite{MR2455943} (and implicitly they are present in earlier works, e.g.~in~\cite{MR584516}). Expander graph sequences (frequently informally referred to as ``expander graphs'') have been studied since at least the 70's in many different branches of mathematics and computer science (see~\cite{MR2247919} for a survey with some historical information).  Both notions (or their close relatives) are broadly used in combinatorics, group theory, ergodic theory, and operator algebras. 

In this article we study the analogues of hyperfinite and expander graph sequences in the context of oriented graphs, particularly directed acyclic graphs. We call these  analogues ``hypershallow'' and ''extender'' graph sequences, respectively.
 Our main result (see Theorem~\ref{thm-main-intro} below) is a stochastic construction of graph sequences which are not hypershallow (we do not know any deterministic construction of such graph sequences). The question whether non-hypershallow graph sequences exist was partially motivated by the techniques presented in~\cite{afshani_et_al:LIPIcs:2019:10586} and in~\cite{MR584516} for obtaining conditional lower bounds in circuit complexity. We will discuss this in Section~\ref{sec-outro}.

Let us now precisely define hypershallow graph sequences and state our main result. 
%~ We chose not to give an unequivocal definition of extender graph sequences in this article, for reasons which we discuss below in Remark~\ref{rem-noext} below. 

\paragraph{Basic conventions} The set of natural numbers is $\NN := 
\{0,1,2,\ldots\}$. We use the shorthand $(X_n)$ for denoting a 
sequence $(X_n)_{n=0}^\infty$. 

A graph is a pair $G=(V,E)$ where $V$ is a non-empty finite set, and $E\subset V\times V$ is a subset which is disjoint from the diagonal. We say that $G$ is \emph{undirected} if $E$ is a symmetric subset of $V\times V$. A \emph{path} of \emph{length $D$} in a graph $G=(V,E)$ is a tuple $(x_0,\ldots, x_D)\in V^{D+1}$, such that for $i<D$ we have either $(x_i,x_{i+1})\in E$ or $(x_{i+1}, x_i)\in E$.

A path $(x_0,\ldots, x_D)$ is \emph{simple} if $x_i\neq x_j$ for $i \neq j$. It is a \emph{directed} path if for all $i<D$ we have  $(x_i,x_{i+1}) \in E(G)$.
A \emph{cycle} is a path $(x_0,\ldots, x_D)$ such that $x_0=x_{D}$. We say that $G$ is a \emph{dag} (which stands for \emph{directed acyclic graph}) if it does not have directed cycles. 

%~ For $x\in V$ we let $\indeg(x;G):=|\{y\in V\colon (y,x)\in E\}|$, and if $(G_n)$ is a sequence of graphs, then we say that $(G_n)$ has \emph{bounded in-degree} if for some $C\in \NN$ and all $n\in \NN$ we have $\max_{v\in V}\indeg(v;G_n) <C$.

We are now ready to define \emph{hypershallow} graph sequences.

\newop{codepth}{codepth}
\begin{definition} 
\begin{enumerate}
\item Let $G$ be a  graph and let $S\subsetneq V(G)$ be a proper subset. We define $\codepth(S;G)\in \NN$ as the maximal $D$ such that there exists a directed simple path $(x_0,\ldots,x_D)$ in $G$ disjoint from $S$. 

\item Let $(G_n)$ be a sequence of dags with uniformly bounded in-degrees. We say 
that $(G_n)$ is \emph{hypershallow} if $\forall \eps>0,\ \exists D\in 
\NN,\ \exists (S_n)$ with $S_n\subsetneq V(G_n)$ and 
$|S_n|<\eps|V(G_n)|$, such that  $\forall n\in \NN$ we have 
$\codepth(S_n;G_n)\le D$.

\end{enumerate}
\end{definition}

\begin{remark}\label{rem_intro1}
Let us take a moment to explicitly state the analogy between the definitions of hypershallow and hyperfinite graph 
sequences.
\begin{enumerate}
\item  We first recall the definition of hyperfinite graph sequences. If $G$ is an undirected graph and $S\subsetneq V(G)$, then  we note that $\codepth(S;G)$ is the maximum of lengths of simple paths disjoint from $S$.

We define a sequence $(G_n)$ of undirected graphs with uniformly bounded degrees to be \emph{hyperfinite} if $\forall \eps>0,\ \exists D\in \NN,\ \exists (S_n)$ with $S_n\subsetneq V(G_n)$ and $|S_n|<\eps|V(G_n)|$, such that $\forall n\in \NN$ we have $\codepth(S_n;G_n)\le D$.

 This is easily seen to be equivalent to the definition of hyperfiniteness in~\cite{MR2455943}.

From this point of view, and with our convention that undirected graphs form a subclass of all graphs, within the class of bounded degree undirected graphs the hypershallow sequences are exactly the same as hyperfinite sequences.

\item Let us explain the choice of the word ``hypershallow'', again by analogy with the word  ``hyperfinite''. One of the simplest classes of undirected graph sequences consists of those sequences $(G_n)$ which have uniformly finite connected components, i.e.~$\exists D$ such that $\forall n$ we have that the connected components of $G_n$ are of size at most $D$. We recall that the expression ``hyperfinite graph sequence'' is meant to suggest that we are dealing with ``the next simplest thing'': informally, a sequence $(G_n)$ is hyperfinite if it is possible to obtain from $(G_n)$ a sequence with uniformly finite connected components by removing an arbitrarily small proportion of vertices from $(G_n)$.

The motivation to use the word ``hypershallow'' is similar. For a dag $G$, let $\depth(G)$ denote the maximum of lengths of directed paths in  $G$. One of the simplest classes of dag sequences with uniformly bounded in-degrees consists of the ``uniformly shallow'' sequences, i.e.~$\exists D$ such that $\forall n$ we have $\depth(G_n)\le D$. The name ``hypershallow graph sequence'' is meant to suggest that we are dealing with ``the next simplest thing'': after removing a small proportion of vertices we get a sequence which is uniformly shallow.\footnote{Uniformly shallow sequences are much more frequently called ``bounded depth sequences''. However, the authors think that ``hypershallow'' {sounds} much better than ``hyper-bounded-depth''.}

\end{enumerate}
\end{remark}

The following definition allows us, informally speaking, to capture ``how badly'' a sequence of graphs fails at being hypershallow.

\begin{definition} 
%~ {\cred
\begin{enumerate}
\item
Let $G$ be a dag, let $\eps, \rho>0$. We say that $G$ is an $(\eps,\rho)$-extender if for every $S\subsetneq V(G)$ with $|S|\le\eps |V(G)|$ we have $\codepth(S;G) \ge \rho$.
\item
Let $(G_n)$ be a sequence of dags with uniformly bounded in-degrees, and let $(\rho_n)$ be a sequence of positive real numbers with $\lim_{n\to\infty} \rho_n = \infty$. We say that $(G_n)$ is a \emph{$(\rho_n)$-extender} sequence if $\lim_{n\to\infty} |V(G_n)| = \infty$ and $\exists \eps>0$, $\exists C>0$, $\forall n\in \NN:$ $G_n$ is an $(\eps, C\rho_{|V(G_n)|})$-extender.
\end{enumerate}
%~ }
\end{definition}

\begin{remark}
It is easy to check that  a sequence $(G_n)$ of dags with uniformly bounded in-degrees is not hypershallow if and only if it contains a subsequence which is a $(\rho_n)$-extender for some $(\rho_n)$ with $\lim_{i\to\infty} \rho_n = \infty$. 
\end{remark}

We are now ready to state our main theorem.

\begin{Theorem}\label{thm-main-intro}
There exists a sequence of  directed acyclic graphs with uniformly bounded degrees which is an $(n^\de)$-extender sequence, with $\de\approx 0.019$.
\end{Theorem}

Our proof of this theorem is probabilistic. The most important part of the proof consists of studying the random graphs $\mathbf G^d_n$ which will be introduced in Section~\ref{sec:example}.  We do not know of a non-probabilistic way of constructing a  non-hypershallow sequence of dags with uniformly bounded degrees.

On the other hand, we can ask how fast the sequence $(\rho_n)$ can grow, provided that there exists a $(\rho_n)$-extender sequence. In this direction we have the following result.
\begin{Theorem}\label{thm-sublinear-intro}
Let $(\de_n)$ be a sequence of numbers in $[0,1]$ such that $\lim_{n\to\infty} \de_n = 1$. If $(G_n)$ is a sequence of directed acyclic graphs with uniformly bounded in-degrees, then $(G_n)$ is not an $(n^{\de_n})$-extender sequence.
\end{Theorem}

\begin{remark}\label{rem-noext}
%~ {\cred %\begin{enumerate}
Theorem~\ref{thm-sublinear-intro} implies, for example, that there are no $(\frac{n}{\log(n)})$-extender sequences. However, we do not know whether there exists an $(n^\de)$-extender sequence for every $\de < 1$. 
%\end{enumerate}
%~ }
\end{remark}

In Section~\ref{sec-prelim} we list some standard definitions and conventions, and we discuss a variant of Pinsker's inequality which involves the Shannon entropy (Proposition~\ref{prop-pinsker}). Pinsker's inequality is the most important external result in our analysis of the random graphs $\mathbf G_n^d$. 

The main part of this article is Section~\ref{sec:example} where we introduce the random graphs $\mathbf G_n^d$ and use them to prove Theorem~\ref{thm-main-intro}.

We conclude this article with Section~\ref{sec-outro}, where we present the proof of Theorem~\ref{thm-sublinear-intro}, and we discuss our initial motivations for studying hypershallow and extender graph sequences which are related to the theory of boolean circuits.

\paragraph{Acknowledgements} We thank one of the anonymous referees for correcting several typos and errors, and for very useful  suggestions for improving  the readability of this article.

We also thank the organisers of the workshop \emph{Measurability, Ergodic Theory and Combinatorics} which took place at Warwick University in July 2019. A very substantial progress on this project happened during that workshop, and we are grateful for inviting both of us and for providing an excellent environment for mathematical collaboration.

Finally, we thank the authors of the blog \emph{Gödel’s Lost Letter and P=NP} for blogging about the paper~\cite{afshani_et_al:LIPIcs:2019:10586} and thereby bringing it to our attention (as well as for many other very interesting posts over the years). This was the starting point for this project.

\section{Preliminaries}\label{sec-prelim}

We use the following conventions. If $n\in \NN$, then $n=\{0,1,\ldots,n-1\}$. If $X$ is a set, then $\Pow(X)$ denotes the power set of $X$, i.e.~the set of all subsets of $X$.
%~ If $X$ is a set, then $2^X$ is the set of all functions from $X$ to $2=\{0,1\}$. This leads to the following notational clash: for $n\in \NN$, the symbol $2^n$ can either denote a number (and hence a set of numbers)  or the set of all functions from $\{0,1,\ldots, n-1\}$ to $\{0,1\}$. We believe that resolving this ambiguity will never cause any difficulty for the reader.

%~ A convenient informal way of thinking about it is that if $k\in 2^n$ then $k$ is both a number smaller than $2^n$ and a function from $n=\{0,\ldots, n-1\}$ to $2=\{0,1\}$, and the translation between the two interpretations is that the binary expansion of a number $k$ smaller than $2^n$ can be thought of as a function from $\{0,\ldots, n-1\}$ to $\{0,1\}$.

\newop{maxoutdeg}{maxoutdeg}
\newop{maxdeg}{maxdeg}

\subsection{Graphs}

\begin{definition}
Let $G=(V,E)$ be a  graph, and let $v,w\in V$.
\begin{enumerate}
\itemsep0pt
\item $\In(v;G):= \{x\in V\colon (x,v)\in G\}$, $\indeg(v;G):= |\In(v;G)|$, 
\item $\Out(v;G) := \{x\in V\colon (v,x)\in G\}$, $\outdeg(v;G):= |\Out(v;G)|$, 
\item $\deg(v;G) := \indeg(v;G) + \outdeg(v;G)$, 
\item $\IN(G):=\{v\in V\colon \indeg(v;G)=0\}$, $\OUT(G):=\{v\in V\colon \outdeg(v;G) = 0\}$, 
\item $\maxindeg(G):= \max_{v\in V}\indeg(v;G)$, $\maxoutdeg(G):= \max_{v\in V}\outdeg(v;G)$,\\ $\maxdeg(G) := \max_{v\in V}\deg(v;G)$,
%~ \item we write $v<_G w$ when there exists a directed path from $v$ to $w$ in $G$ and we write $x\le_G y$ if $x=y$ or $x<_G y$.
\end{enumerate}
\end{definition}

\begin{definition} Let $(G_n)$ be a sequence of graphs. We say that $(G_n)$ has, respectively, \emph{bounded degree}, \emph{bounded in-degree}, or \emph{bounded out-degree}, if, respectively,\\ $\max_{n\in \NN}\maxdeg(G_n)< \infty$, $\max_{n\in \NN}\maxindeg(G_n)< \infty$, or $\max_{n\in \NN}\maxoutdeg(G_n)< \infty$.
\end{definition}

\subsection{Probability}
\begin{definition}
\begin{enumerate}
\item If $\mu$ is a probability measure on $\NN$, then we also use the symbol $\mu$ for the function $\NN\to \RR$ which sends $k\in\NN$ to $\mu(\{k\}$) (so in particular we can write  $\mu(k)$ instead of $\mu(\{k\}))$, and we let
$$
    H(\mu) := -\sum_{i\in\NN} \mu(i)\log(\mu(i)),
$$
where by convention $0\log(0)=0$. 

\item A \emph{random variable} on a standard probability space $(X,\mu)$ with values in a standard Borel space $Y$ is a Borel function $f\colon X \to Y$. The \emph{law} of $f$ is the push-forward measure $f^\ast(\mu)$ on $Y$, i.e.~for $U\subset Y$ we let  $f^\ast(\mu)(U):= \mu(f^{-1}(U))$. 
%~ If $f$ is $\NN$-valued and $\al$ is the law of $f$ then we will consider $\al$ as a measure on $\NN$.
\item If $f$ is an $\NN$-valued random variable and $\al$ is its law, then we define $H(f):=H(\al)$. 
\item If $f$ and $g$ are random variables with values in a standard Borel space $Z$, then we define a new random variable $f\sqcup g$ with values in $Z$ by, informally speaking, choosing between $f$ and $g$ with probability $\frac{1}{2}$.

 Formally, suppose that $f$ and $g$ are defined on $(X,\mu)$ and $(Y,\nu)$, respectively. The probability space on which $f\sqcup g$ is defined is $(X\sqcup Y,\om)$, where $\om$ is the unique measure on $X\sqcup Y$ such that $\om(U) = \frac{\mu(U)}{2}$ when $U\subset X$ and $\om(U) = \frac{\nu(U)}{2}$ when $U\subset Y$. We let $f\sqcup g(x):= f(x)$ for $x\in X\subset X\sqcup Y$ and $f\sqcup g(y) := g(y)$ for $y\in Y\subset X\sqcup Y$. 

\item For $\al\colon \NN \to \RR$ we let $\|\al\|_1 := \sum_{i\in \NN} |\al(i)|$. 
\end{enumerate}
\end{definition}

\begin{lemma}\label{lem-entropy-basics}
\part If $f$ and $g$ are random variables with values in the same space $Y$, with laws $\al$ and $\be$ respectively, then the law of $f\sqcup g$ is $\frac{\al+\be}{2}$.
\part If $f$ is a random variable with values in $\{0,1,\ldots, k-1\}$ then $H(f)\le \log(k)$. 
\trap
\end{lemma}
\begin{proof} \part Follows directly from the definitions.
\part See e.g.~\cite[§2.3]{galvin2014tutorial}
\trap
\end{proof}

The main point of the following proposition is contained in its second item. Informally, it allows us to say the following: if $f$ and $g$ are $\NN$-valued random variables with laws $\al$ and $\be$ respectively,  and $(x,y)\in \NN^2$ is chosen according to the law $\al\times \be$, then either it is roughly as probable that $x>y$ as it is that $y>x$, or the entropy of $f\sqcup g$ is substantially larger than the average of the entropies of $f$ and $g$.

\begin{proposition}\label{prop-pinsker}
Let $f$ and $g$ be $\NN$-valued random variables with laws $\al$ and $\be$, respectively.
\begin{enumerate}
\item We have $2H(f\sqcup g) - H(f) - H(g) \ge 0$.
\item We have 
\begin{equation}\label{eq-p}
    \al\times \be\left(\{(x,y)\in \NN^2\colon x> y\}\right)  \le \frac{1}{2} +  2\sqrt{2H(f\sqcup g)- H(f)-H(g)}.
\end{equation}
\end{enumerate}
\end{proposition}

\begin{proof}
By the previous lemma we have that the law of $f\sqcup g$ is $\ga:= \frac{\al+\be}{2}$. As such the first item follows from Jensen's inequality.

The second item is a simple corollary of Pinsker's inequality (see e.g.~\cite[Theorem 2.16]{MR2319879} for the statement and proof of Pinsker's inequality). To derive it, we start by stating the following two special cases of Pinsker's inequality:
$$
    \|\al-\ga\|_1^2 \le 2D(\al \|\ga)
$$
and
$$
    \|\be-\ga\|_1^2 \le 2D(\be \|\ga),
$$
where 
$$
    D( \al \| \ga) := \sum_{i\in \NN} \al(i)\log\left(\frac{\al(i)}{\ga(i)}\right),
$$
and similarly for $D(\be\|\ga)$. By convention we set $0\log(0) = 0\log(\frac{0}{0}) =0$ in the definitions of $D(\al\|\ga)$ and $D(\be\|\ga)$.

Noting that $\|\al -\ga\|_1 = \|\be - \ga\|_1$, summing the two inequalities above gives
$$
    2\|\be -\ga\|^2_1 = 2\|\al-\ga\|^2_1 \le 2D(\al \|\ga) + 2D(\be \|\ga).
$$
A direct computation shows that $D(\al \|\ga) + D(\be \|\ga) = 2H(\ga) -H(\al) - H(\be)$, so together with the triangle inequality we deduce that  
\begin{equation}\label{eq-q}
\|\al -\be\|_1 \le 2\sqrt{2H(\ga) - H(\al)-H(\be)}.
\end{equation}

On the other hand, the left-hand side of~\eqref{eq-p} is equal to
\begin{align*}
\sum_{i,j\in \NN\colon i<j} \al(i)\be(j) &\le \sum_{i,j\in \NN\colon i<j} \al(i)(\al(j)+ \|\al-\be\|_1) 
\\
& \le \al\times \al \left(\{(x,y)\in \NN^2\colon x> y\}\right)+ \|\al-\be\|_1  
\\
\le \frac{1}{2} + \|\al-\be\|_1
\end{align*}
 which together with \eqref{eq-q} finishes the proof.
\end{proof}

\section{Existence of non-hypershallow sequences}\label{sec:example}

In this section we will describe a probabilistic construction of non-hypershallow sequences of dags. They will be in fact $n^\de$-expander sequences for $\de\approx 0.019$. 

We will construct a sequence of random graphs $\mathbf G_n^d$ which asymptotically almost surely forms, after small modifications, an $n^\de$-extender sequence. The graphs $\mathbf G_n^d$ will be essentially defined as follows. The vertices are $\{1, 2, ..., n\}$ and for every $i < j$, we add an edge $(i, j)$ independently with probability proportional to $\frac{1}{j-i}$. In order to simplify the proof, we will slightly change the probabilities when we define $\mathbf G_n^d$ in Subsection~\ref{subsec-rgdef}.

We start with the definition and discussion of \emph{depth functions} in Subsection~\ref{subsec-dfun}, as they provide a convenient way of characterising the property of being an $(\eps, \rho)$-extender, which will be crucial in the analysis of the random graphs $\mathbf G_n^d$  in Subsection~\ref{subsec-analysis}.

\subsection{Depth functions}\label{subsec-dfun}

Given a  graph $G$ and $S\subset V(G)$, we can associate to it 
a function which ``measures the maximal distance to $S\cup \Out(G)$''. More precisely, we define $\de_S\colon V(G) \to \NN$ by setting $\de_S(x)=0$ when $x\in S\cup\Out(G)$, and for $x\notin S\cup\Out(G)$ we let $\de_S(x)$ to be the maximal $l\in \NN$ for which there exists a directed simple 
path $x_0,\ldots, x_l$ with $x_0=x$, $x_l\in S\cup \Out(G)$, and $x_i\notin S$ when 
$0\le i<l$. Let us start by abstracting some properties of $\de_S$ into the notion of a \emph{depth function} as 
follows. 
%~ For technical reasons it is convenient to work with arbitrary 
%~ directed graphs.

\begin{definition}\label{def-depth-fun} Let $G$ be a  graph.
\begin{enumerate} 
\item A \emph{depth function for $G$} is a function $f\colon V(G) \to \NN$ such that the following conditions hold:
\begin{enumerate}
\item For every $(a,b)\in E(G)$ we have either $f(a)>f(b)$ or $f(a)=0$
\item For every $a\in V(G)$ such that $f(a)\neq 0$ there exists $b\in V(G)$ such that $(a,b)\in E(G)$ and $f(b)=f(a)-1$. 
\end{enumerate}

\item Let $\eps>0$ and let $\rho\in \NN$. An \emph{$(\eps,\rho)$-depth function for $G$} is a depth function $f$ for $G$  such that for all $v\in V(G)$ we have $f(v) \le \rho$  and $|f^{-1}(0)\setminus \Out(G)| \le \eps |V(G)|$. 
\end{enumerate}
\end{definition}

\begin{example}\label{ex-delta-s} It is straightforward to verify that if $S\subset V(G)$ then  
$\de_S$ is a 
$$
    \left(\frac{|S\setminus \Out(G)|}{|V(G)|},\,\, \codepth(S;G)+1\right)
$$
-depth function.
\end{example}

\begin{lemma}
If $f$ is a $(\eps,\rho)$-depth function, then $\codepth(f^{-1}(0)\setminus \Out(G);G) \le\rho$.
\end{lemma}
\begin{proof}
Let $(x_0,x_1,\ldots, x_k)$ be a simple path disjoint from $f^{-1}(0)\setminus \Out(G)$. We have $f(x_0) \le \rho$ and $f(x_{i+1})< f(x_i)$  for all $i<k$, by 1(a) in Definition~\ref{def-depth-fun}. The only vertex, in any simple path, which can be in $\Out(G)$ is the last vertex, so we deduce that $x_{k-1} \notin f^{-1}(0)$, i.e.~$f(x_{k-1}) >0$. This shows that $k-1<\rho$, and thus $k\le \rho$. This shows that any simple path disjoint from $f^{-1}(0)\setminus \Out(G)$ has length at most $\rho$, which proves the lemma.
\end{proof}

This lemma allows us to characterise extender graphs as follows.

\begin{corollary}
Let $\eps,\rho>0$ and let $G$ be a directed acyclic graph. 
\part If $G$ is an $(\eps,\rho)$-extender then there are no $(\eps,\rho)$-depth functions for $G$.
\part If there are no $(\eps,\rho+1)$-depth functions for $G$ then $G$ is an $(\eps,\rho)$-extender. 
\trap
\end{corollary}
\begin{proof}
\part Let $\eps'>0$ and suppose that $f$ is a $(\eps',\rho)$-depth function for $G$. 

By the previous lemma we have $\codepth (f^{-1}(0)\setminus \Out(G);G) \le \rho$. Therefore, since we assume that $G$ is an $(\eps,\rho)$-extender, we have $|f^{-1}(0)\setminus \Out(G)| > \eps |V(G)|$. By the definition of being a $(\eps',\rho)$-depth function we have $|f^{-1}(0)\setminus \Out(G)|  \le \eps' |V(G)|$, which shows $\eps'>\eps$.

 %~ by Definition~\ref{def-dag2}, we have $\eps'>\eps$, which show that there is no $(\eps,\rho)$-depth function for $G$.

\part  Let $S\subset V(G)$ be a set with $|S|\le \eps |V(G)|$. Then by Example~\ref{ex-delta-s}, we have that $\de_S$ is a $(\eps, \codepth(S; G)+1)$-depth function. Since we assume that there are no $(\eps,\rho+1)$-depth functions,  we deduce $\codepth(S;G)+1> \rho+1$, and hence $\codepth(S;G) >\rho$. This shows that $G$ is an $(\eps,\rho)$-extender and finishes the proof.
\trap
\end{proof}

It will be useful to restate the above corollary for the case of graph sequences.

\begin{corollary}\label{cory-depth-ext}
Let $(G_n)$ be a bounded in-degree sequence of directed acyclic graphs and let $(\rho_n)$ be a sequence of positive real numbers with $\lim_{n\to \infty} \rho_n = \infty$. The following conditions are equivalent.
\begin{enumerate}
\item[a)] The sequence $(G_n)$ is a $(\rho_n)$-extender sequence
\item[b)] There exists $C>0$, $\eps>0$ such that for all $n\in\NN$  we have that $G_n$ does not admit a $(\eps,C\cdot \rho_{|V(G_n)|})$-depth function.
\end{enumerate}\qed
\end{corollary}

One of the steps in the proof of Theorem~\ref{thm-main-intro}, Proposition~\ref{prop-c}, requires bounding the number of depth functions on a graph. We finish this subsection with a lemma which is used to count depth functions in the proof of Proposition~\ref{prop-c}. First we need the following definition. 
\begin{definition}
For $D\subset E(G)$ we define $\de'_D\colon V(G) \to \NN\cup \{\infty\}$ by setting $\de'_D(v)$ to be equal to the maximal length of a directed simple path in the graph  $(V(G),D)$ which connects $v$ to a vertex in $\Out((V(G),D))$.
\end{definition}

In other words, $\de'_D$ is the ``standard'' depth function for the graph $(V(G),D)$. While it is not true that for every $D\subset E(G)$ we have that  $\de'_D$ is a depth function for the graph $G$, the following lemma shows in particular that for every depth function $f$ we can find $D$ such that $f=\de'_D$.

\begin{lemma}\label{lem-de-prime} Let $G$ be a  graph and $f$ be a depth function for $G$. Then there exists $D\subset E(G)$ such that $|D|\le |V(G)|$ and $\de'_D = f$.
\end{lemma}
\begin{proof}
By Condition 1b) of Definition~\ref{def-depth-fun}, there exists a function $n\colon V(G)\setminus f^{-1}(0) \to V(G)$ such that for every $v\in  V(G)\setminus f^{-1}(0)$ we have $f(n(v)) = f(v)-1$.  We let $D:=\{(v,n(v))\colon v\in V(G)\setminus f^{-1}(0)\}$. It is straightforward to check that $D$ has the desired properties.
\end{proof}

\subsection{Definition and basic properties of the random graphs $\mathbf G_n^d$}\label{subsec-rgdef}

\newcommand{\EE}{\mathbf E}
In this article a \emph{random graph} is a pair $(V, \EE)$ where $V$ is a non-empty finite set and $\EE$ is a random variable with values in $\Pow(V\times V)$ such that $\EE$ is disjoint from the diagonal in $V\times V$ almost surely.

For $n\in \NN$ let $\ZZ_n:= \ZZ/n\ZZ$. For $a,b\in \ZZ_n$ we write $a>b$ if $a = \bar a + n\ZZ$, $b = \bar b + n\ZZ$, where $\bar a, \bar b\in \{0,1,\ldots, n-1\}$ and $\bar a >\bar b$.  We also let $R(n):=\lfloor\log(n)\rfloor$.

We start by defining a random variable $J_n$ with values in $\ZZ_n\times \ZZ_n$, as follows. We first choose $v\in \ZZ_n$ uniformly at random, then we choose $r\in R(n)$ uniformly at random, and we  choose $(x,y)$ uniformly at random in 
$$
\{v,v+1,\ldots, v+2^r-1\} \times \{v+2^r, v+2^r+1,\ldots, v+2^{r+1}-1\}\subset \ZZ_n\times \ZZ_n. 
$$
The law of $J_n$ will be denoted with $\iota_n$.

Now for $d, n\in \NN$ we define a random graph $\mathbf G_n^d$ as follows: we let $V(\mathbf G_n^d) := \ZZ_n$, and the random variable $\EE_n^d$ with values in $\Pow(\ZZ_n\times \ZZ_n)$ is defined by choosing $dn$ elements of $\ZZ_n\times \ZZ_n$ independently at random according to the law $\iota_n$. This finishes the definition of $\mathbf G_n^d$.

Let us note that $\mathbf G_n^d$ is typically neither a dag nor of bounded degree, but the following lemma implies that with high probability $\mathbf G_n^d$ becomes a bounded degree dag after removing a small amount of vertices.

\begin{lemma}\label{lem-pre} Let $\eps>0$, $d,n,\De\in \NN$, and let $\EE:= \EE_n^d$. We have
%~ \begin{enumerate}
%~ \item For all $k,n,d\in \NN$ and $v\in \ZZ_n$ we have 
%~ $$
    %~ \Pr_{E_n^d}(\indeg(v; (\ZZ_n, E_n^d))>k) < \frac{d^k}{k!}
%~ $$

%~ For every $\eps>0 ,d\in \NN$ there exists $k\in\NN$ such that for every $n\in \NN$  we have 
\begin{equation}\label{eq-fa}
\Pr_{\EE} (|\{v\in \ZZ_n \colon \deg(v; (\ZZ_n,\EE)) \ge 2\De \}| \ge \eps \cdot n ) \le \frac{2d^\De}{\De!\cdot \eps}
\end{equation}
and
%~ For every $\eps>0, d\in \NN$ there exists $n_0\in \NN$ such that for $n>n_0$ we have
\begin{equation}\label{eq-fb}
    \Pr_{\EE}( |\{(x,y)\in \EE \colon x>y \}| \ge \eps\cdot d n)  \le \frac{2d}{\eps\cdot R(n)}
\end{equation}

\end{lemma}
\begin{proof}
Note that we have the following $(\ZZ_n\times \ZZ_n)$-valued random variable $J'_n$ whose law is the same as the law of $J_n$, i.e.~it is equal to $\iota_n$.  We choose $(v,r)\in \ZZ_n \times R(n)$ uniformly at random, then we choose $(a,b)\in 2^r\times 2^r$ uniformly at random and we choose the edge $(v,v+2^r+b-a)$.  

Therefore, if we fix $v\in \ZZ_n$, then
$$
    \Pr_{\EE}(\outdeg(v; (\ZZ_n, \EE))\ge \De) \le {nd \choose \De} \frac{1}{n^\De} \le \frac{d^\De}{\De!},
$$
where the last inequality is obtained by writing 
$$
{nd \choose \De} \frac{1}{n^\De}  = \frac{nd\cdot\ldots\cdot (nd-\De+1)}{\De!} \frac{1}{n^\De} \le
\frac{(nd)^\De}{\De!} \frac{1}{n^\De} = \frac{d^\De}{\De!}.
$$
Similarly, for a fixed $v\in \ZZ_n$, we have 
$$
    \Pr_{\EE}(\indeg(v; (\ZZ_n, \EE))\ge \De) \le \frac{d^\De}{\De!},
$$
and hence 
$$
    \Pr_{\EE}(\deg(v; (\ZZ_n, \EE))\ge 2\De) \le \frac{2d^\De}{\De!}.
$$

Now by linearity of expectation we have
$$
\E_{\EE} (|\{v\in \ZZ_n \colon \deg(v; (\ZZ_n,\EE))\ge 2\De \}|) = \sum_{v\in \ZZ_n}     \Pr_{\EE}(\deg(v; (\ZZ_n, \EE))\ge 2\De),
$$
and the right-hand side is bounded from above by $\frac{2nd^\De}{\De!}$. Thus, by Markov's inequality we have
$$
\Pr_{\EE} (|\{v\in \ZZ_n \colon \deg(v; (\ZZ_n,\EE))\ge 2\De \}| \ge \eps \cdot n ) <\frac{2nd^\De}{\De! \eps n} = \frac{2d^\De}{\De!\eps},
$$
which finishes the proof of~\eqref{eq-fa}.

In order to prove~\eqref{eq-fb}, we start by bounding $\iota_n(\{(x,y)\in \Z_n\times \Z_n\colon x>y\})$ from above. By the definition of $J_n'$, the only way in which $J_n'$ might take a value $(x,y)$ with $x>y$ is when we start by choosing $(v,r)\in \ZZ_n \times R(n)$ such that $v>n-2^{r+1}$. As such we have
$$
    \iota_n(\{(x,y)\in \Z_n\times \Z_n\colon x>y\}) \le \frac{1}{R(n)}\sum_{r< R(n)} \frac{1}{n}|\{a\in \ZZ_n\colon a>n-2^{r+1}\}|,
$$
which is bounded from above by 
$$
\frac{1}{nR(n)} \sum_{r<R(n)} 2^{r+1} < \frac{1}{nR(n)} 2^{R(n)+1} \le \frac{2n}{nR(n)} = \frac{2}{R(n)}
$$

Therefore, we have
$$
\E_{\EE}( |\{(a,b)\in \EE \colon a>b \}|)  < \frac{2nd}{R(n)}
$$
and Markov's inequality again gives us the desired bound.
\end{proof}

\subsection{Construction of an $n^\de$-extender sequence from $\mathbf G_n^d$}\label{subsec-analysis}

The key lemma which we need is the following.

\begin{Lemma} \label{lemma:l}
Let $d,n\in \NN$, let $\eps>0$, let $k:=\lfloor n^{\eps^3}\rfloor$, and  let $l\colon \Z_n \to k$. We have
\begin{equation}\label{eq-key}
\iota_n(\{(a,b)\in \ZZ_n\times \ZZ_n\colon l(a)>l(b)\}) < \frac{1}{2} + 4\eps 
\end{equation}
\end{Lemma}

Let us discuss the intuition behind the proof of Lemma~\ref{lemma:l}. First, let us discuss the meaning of the left-hand side of~\eqref{eq-key}. We first choose $(v,r)\in \ZZ_n\times R(n)$ uniformly at random, then we look at the distribution of  
$l(v)$, $l(v+1)$, ..., $l(v+2^r-1)$ on the one side and the distribution of $l(v+2^r)$, $l(v+2^r+1)$, ..., $l(v+2^{r+1}-1)$ on the other side. We sample an element of $\{0,\ldots, k-1\}$ from the first distribution and an element of $\{0,\ldots, k-1\}$ from the second distribution. Then the left-hand side of~\eqref{eq-key} is the probability that the first element is larger than the second element.

If the distribution of $l(v)$, $l(v+1)$, ..., $l(v+2^r-1)$ is very close to the distribution of $l(v+2^r)$, $l(v+2^r+1)$, ..., $l(v+2^{r+1}-1)$, then for a random edge between the two vertex sets, $l$ increases or decreases with approximately the same probability. But if the two distributions are not very close, then the entropy of the distribution of the union $l(v)$, $l(v+1)$, ..., $l(v+2^{r+1}-1)$ is larger than the average of the two entropies (this statement is formalised by Proposition~\ref{prop-pinsker}). 

As the entropy of the distribution of $l(v),\ldots, l(v+2^{R(n)})$ is bounded from above by $\log(k)$ (by Lemma~\ref{lem-entropy-basics}b)), it should be clear that when we choose $k$ sufficiently small, then for a fixed $v\in \ZZ_n$ there will be only a small amount of $r$'s for which the distribution of $l(v)$, $l(v+1)$, ..., $l(v+2^r-1)$ and $l(v+2^r)$, $l(v+2^r+1)$, ..., $l(2^{r+1}-1)$  is very different. 

%~ Our choice $k= \lfloor n^{\eps^3}\rfloor $ is the largest $k$ which works easilly for the above strategy.
  
%~ and this latter case must rarely happen. 

\begin{proof}[Proof of Lemma~\ref{lemma:l}]
For $v\in \ZZ_n$, $r<R(n)$, let $X_{v,r}$ denote the restriction of $l$ to $[v,v+2^r-1]\subset \Z_n$, and let $Y_{v,r}$ denote the restriction of $l$ to $[v+2^r,v+2^{r+1}-1]$. We consider $X_{v,r}$ and $Y_{v,r}$ as $k$-valued random variables.
 %~ and let $Z_{v,r}$ denote the restriction of $l$ to $[v, v+2^{r+1}-1]$. 

Note that $X_{v,r+1} = X_{v,r} \sqcup Y_{v,r}$. As such, by the first item of Proposition~\ref{prop-pinsker}, for all $v,r$ we have
$$
2 \cdot H(X_{v,r+1}) - H(X_{v,r}) -H(Y_{v,r}) \ge 0.
$$

On the other hand, we have $\E_{v,r} (H(X_{v,r})) = \E_{v,r} ( H(Y_{v,r}))$, where $(v,r)$ is chosen uniformly at random from $\ZZ_n\times R(n)$. Hence
$$
\E_{v,r} (2 \cdot H(X_{v,r+1}) - H(X_{v,r}) -H(Y_{v,r})) = 2\E_{v,r}(H(X_{v,r+1})) - 2\E_{v,r}(H(X_{v,r})),
$$
and so
\begin{align*}
\E_{v,r} (2 \cdot H(X_{v,r+1}) - H(X_{v,r}) -H(Y_{v,r})) 
&= \frac{2}{R(n)} \E_v(\sum_{r<R(n)} H(X_{v,r+1}) - H(X_{v,r}))
\\
&= \frac{2}{R(n)} \E_v ( H(X_{v,R(n)})-H(X_{v,0}))
\\
&\le \frac{2\log(k)}{R(n)}.
\end{align*}
Now Markov's inequality shows that 
$$
\Pr_{v,r}  (2 \cdot H(X_{v,r+1}) - H(X_{v,r}) -H(Y_{v,r}) \ge \eps^2) \le  \frac{2\log(k)}{\eps^2 \cdot R(n)} \le  2\eps.
$$

By the second item of Proposition~\ref{prop-pinsker}, if for some $r,v,\eps$ we have  $H(X_{v,r+1}) - H(X_{v,r}) -H(Y_{v,r}) < \eps^2$, then $\Pr_{x,y}(X_{v,r}(x) > Y_{v,r}(y)) \le \frac{1}{2} +2\eps$. Thus by the definition of $\iota_n$, we have
\begin{align*}
\iota_n(\{(a,b)&\in \ZZ_n\times \ZZ_n\colon l(a)>l(b)\}) <
\\
&< \frac{1}{2} + 2\eps +  \Pr_{v,r}  (2 \cdot H(X_{v,r+1}) - H(X_{v,r}) -H(Y_{v,r}) \ge \eps^2) 
\\
&\le \frac{1}{2}+4\eps,
\end{align*}
which finishes the proof.
\end{proof}

\begin{proposition}\label{prop-c}
Let $d,n\in \NN$, $d\ge 3$, let $\EE:= \EE_n^d$, let $\eps\in(0,1)$, and let $k:=\lfloor n^{\eps^3}\rfloor$,
\begin{equation}\label{eq-pp}
    \Pr_\EE(\text{$(\ZZ_n,\EE)$ admits a $(\eps,k)$-depth function}) <  2^{H(\frac{1}{d})dn} (\frac{1}{2} + 4\eps)^{(d-1)n},
\end{equation}
where for $x\in (0,1)$ we set $H(x) = -x\log(x) - (1-x)\log(1-x)$.
\end{proposition}
\begin{proof}
Clearly it is enough to show that 
\begin{equation}\label{eq-dd}
    \E_\EE(\text{number of $(\eps,k)$-depth functions for $(\ZZ_n,\EE)$})<  2^{H(\frac{1}{d})dn} (\frac{1}{2} + 4\eps)^{(d-1)n}.
\end{equation}

By Lemma~\ref{lem-de-prime},  the left-hand side of~\eqref{eq-dd} is 
bounded above by
\begin{align}\label{eq-exp12}
    \E_\EE\Big(\big| &\big\{D\subset \EE\colon  
\\
&\text{ $|D|\le  n$ and $\de'_D$ is an  $(\eps,k)$-depth function for $(\ZZ_n,\EE)$}\big\}\big|\Big).\nonumber
\end{align}
\newop{Set}{Set}

Given $I\subset dn$, let $\Set_I\colon (\ZZ_n\times \ZZ_n)^{dn} \to \Pow(\ZZ_n\times \ZZ_n)$ be defined by 
$$
    \Set_I(x_0,\ldots, x_{dn-1}) := \{x_i\colon i\in I\},
$$ 
and let $\Set:= \Set_{dn}$.  Furthermore if $G$ is a graph and $D\subset E(G)$ is such that $\de'_D$ is a $(\eps,k)$-depth function, then let us say that $D$ is an \emph{$(\eps,k)$-depth set} for $G$. 

Recall that the law of $\EE$ is the push-forward of $\iota^{dn}_n$ through the map $\Set$. As such, we deduce that~\eqref{eq-exp12} is bounded above by
\begin{align}\label{eq-ww}
    \sum_{I\subset dn\colon |I|\le n} \iota_n^{dn}\Big(&\big\{(e_i)_{i\in dn} \in (\ZZ_n\times \ZZ_n)^{dn}\colon
\\
&\text{$\Set_I((e_i)_{i\in dn})$ is a $(\eps,k)$-depth set for $(\ZZ_n, \Set((e_i)_{i\in dn})$}\big\}\Big)\nonumber
 \end{align}

Let us first estimate the number of summands in~\eqref{eq-ww}. Recall that for $0<\al\le\frac{1}{2}$ and $m\in \NN$ we have
 $\sum_{i\le \al m} {m \choose i} \le 2^{H(\al) m}$ (see e.g.~\cite[Theorem 3.1]{galvin2014tutorial}). Since $\eps<1$ and $d\ge 3$, we see that the number of summands in~\eqref{eq-ww} is therefore at most $2^{H(\frac{1}{d})dn}$.

To estimate each summand, let us fix $I\subset dn$, and let us fix $(e_i)_{i\in I}\in (\ZZ_n\times \ZZ_n)^I$. Let $D:= \{e_i\colon i \in I\}$ and let $l$ be the depth function $\de'_D$ on the graph $(\ZZ_n,D)$.  The probability that $l$ will still be a depth function after we add $nd-|I|$ remaining edges is, by Lemma~\ref{lemma:l}, at most 
$$
(\frac{1}{2} + 4\eps)^{dn-|I|} \le (\frac{1}{2} + 4\eps)^{(d-1)n}.
$$
As such, we have that~\eqref{eq-ww} is bounded above by
$$
2^{H(\frac{1}{d})dn} (\frac{1}{2} + 4\eps)^{(d-1)n},
$$
and hence~\eqref{eq-dd} and~\eqref{eq-pp} hold true. This finishes the proof.
\end{proof}

We are now ready to prove Theorem~\ref{thm-main-intro}. Clearly it follows from the following theorem.

\begin{Theorem} \label{theorem:NHS}
Let $\eps\in(0,\frac{1}{8})$. Then there 
exists a bounded degree sequence  $(H_n)$ of directed acyclic graphs which is an $\left(n^{\eps^3}\right)$-extender sequence.
\end{Theorem}
\begin{proof}
Let $\de>0$ be such that $p:= \frac{1}{2} + 4(\eps +2\de) <1$.  Let $d$ be such that 
\begin{equation}\label{eq-tr}
2^{H(\frac{1}{d})} \cdot (\frac{1}{2} + 4(\eps+2\de))^{\frac{d-1}{d}}< 1.
\end{equation}
It is possible to choose such $d$ since $H(x)\to 0$ as $x\to 0$.  Let $\De\in \NN$ be such that 
$\frac{2d^\De}{\De!\cdot \de}<\de$, and let $n_0$ be such that for $n>n_0$ we have  $\frac{2d^2}{\de R(n)} <\de$ and 
$$
    2^{H(\frac{1}{d})dn} (\frac{1}{2} + 4(\eps+2\de))^{(d-1)n}<1-2\de,
$$ 
which is possible by~\eqref{eq-tr}. 

Therefore, by Proposition~\ref{prop-c}, we have for $n>n_0$ that 
$$
    \Pr_\EE(\text{$(\ZZ_n,\EE)$ admits a $(\eps+2\de,n^{(\eps+2\de)^3})$-depth function}) < 1-2\de.
$$ 
Furthermore, by Lemma~\ref{lem-pre} we have  
$$
    \Pr_{\EE} (|\{v\in \ZZ_n \colon \deg(v; (\ZZ_n,\EE))\ge 2\De \}| \ge \de \cdot n ) \le \frac{2d^\De}{\De!\cdot \de}<\de
$$
and
$$
    \Pr_{\EE}( |\{(a,b)\in \EE \colon a>b \}| \ge\frac{\de}d\cdot dn=\de n)  \le \frac{2d}{\frac{\de}{d} R(n)}<\de.
$$
As such, by the union bound, we get for each $n>n_0$ a  graph $G_n$ with $V(G_n)=\ZZ_n$ such that $G_n$ does not admit a $(\eps+2\de,n^{(\eps+2\de)^3})$-depth function, and furthermore 
$$
    |\{v\in \ZZ_n \colon \deg(v; G_n)\ge 2\De \}| \le \de \cdot n 
$$
and 
$$
    |\{(a,b)\in E(G_n) \colon a>b \}| \le\de\cdot n.
$$

Let $B\subset \ZZ_n$ be the union of
$$
    \{v\in \ZZ_n\colon \deg(v;G_n)\ge 2\De\}
$$
and 
$$
    \{a\in \ZZ_n\colon \exists b\in \ZZ_n \suchthat (a,b)\in E(G_n) \text{ and } a>b\}.
$$ 
We let  $H_n$ be the subgraph of $G_n$ induced by the set of vertices $\ZZ_n\setminus B$.
Clearly $H_n$ is a sequence of bounded degree 
dags, and since $|B|\le 2\de n$, we see that $H_n$ does not admit a $(\eps,n^{(\eps+2\de)^3})$-depth function, and hence it also does not admit a $(\eps,n^{\eps^3})$-depth function. By Corollary~\ref{cory-depth-ext}, this finishes the proof.
\end{proof}

\section{Final remarks}\label{sec-outro}

Let us proceed with the proof of Theorem~\ref{thm-sublinear-intro}.  Clearly Theorem~\ref{thm-sublinear-intro} follows from the following proposition.

\begin{proposition}
Let $(G_n)$ be a sequence of bounded in-degree directed acyclic graphs and let $(\de_n)$ be a sequence of real numbers in the interval $(0,1]$ such that $\lim_{n\to \infty} \de_n =1$. For every $\eps>0$ there exists a sequence $(S_n)$ with $S_n\subsetneq V(G_n)$ such that $|S_n| <\eps|V(G_n)|$ and 
\begin{equation}\label{todo34}
    \lim_{n\to \infty} \frac{\codepth(S_n;G_n)}{|V(G_n)|^{\de_n}} = 0
\end{equation}
\end{proposition}
\begin{proof}Let $m_n := |V(G_n)|$, let $d\in \NN$ be such that $\maxindeg(G_n) \le d$ for all $n\in \NN$, and let us fix $\eps>0$. Since the graphs $G_n$ are dags, we may assume that $V(G_n)=\{0,\ldots, m_n-1\}$ and that $(x,y)\in E(G_n)$ implies $x<y$.

Let us informally describe our strategy for constructing the sets $S_n$: first we will include in $S_n$ all vertices adjacent to an edge of length between $m_n^{c_n}$ and $m_n^{c_n+\eps}$ for a suitable $c_n$. This way any directed path disjoint from $S_n$ will have edges of length either less than $m_n^{c_n}$  (``short edges'') or larger than $m_n^{c_n+\eps}$ (``long edges'').  

The number of the long edges in a directed path is at most $m_n^{1-c_n-\eps}$.  To bound the number of short edges, we will also include in $S_n$ all vertices which are congruent to at most $m_n^{c_n}$ modulo $\frac{m_n^{c_n}}{\eps}$. This way any path disjoint from $S_n$ and consisting only of short edges must be contained in an interval of length $\frac{m_n^{c_n}}{\eps}$, and so in particular its length is at most $\frac{m_n^{c_n}}{\eps}$. 

These bounds on the total number of long edges, and the maximal number of consecutive short edges allow us to obtain the desired bound on $\codepth(S_n;G_n)$. Let us now make it precise. 

Since $|E(G_n)|\le dm_n$, by the pigeon hole principle we may find $c_n\in [0,1)$   %]$
such that the set
$$
   X_n:= \{(x,y)\in E(G_n)\colon y-x \in [m_n^{c_n}, m_n^{c_n+\eps})\}
$$
has cardinality at most $\eps \cdot d m_n$. Let $A:=\lfloor \frac{m_n^{c_n}}{\eps}\rfloor $, let $B:=\lfloor m_n^{c_n} \rfloor$, and let 
$$
    Y_n: =\{ x\in V(G_n)\colon x \equiv k \operatorname{mod} A \text{ for some $k$ with $0\le k < B$}\}.
$$
Finally we let
$$
    S_n := \{x\in V(G_n)\colon \exists y\in V(G_n) \suchthat (x,y)\in X_n\} \cup Y_n.
$$
Clearly we have $|S_n| \le |X_n| +|Y_n| \le \eps dm_n + \eps m_n= \eps(d+1)m_n$. Thus, since $\eps$ was arbitrary, in order to finish the proof it is enough to argue that~\eqref{todo34} holds. 

In order to estimate $\codepth(S_n;G_n)$ let us fix $n\in \NN$ and let 
\begin{equation}\label{path}
    (x_0,\ldots,x_{l})
\end{equation} 
be a directed path in $G_n$ disjoint from $S_n$. By the definition of $X_n$, and since $S_n$ contains all starting vertices of edges in $X_n$,  we see that for all $i< l$ we have either $x_i-x_{i+1}< m_n^{c_n}$ or $x_i-x_{i+1} \ge m_n^{c_n+\eps}$. 

Let $(x_{j},x_{j+1},\ldots, x_{j+M})$ be a maximal subpath of~\eqref{path}  such that for all $k<M$ 
we have $x_{j+k+1} -x_j < m_n^{c_n}$. Since $Y_n\subset S_n$, we see 
that $M\le\frac{m_n^{c_n}}{\eps}$. On the other hand the maximal number of edges in the path~\eqref{path} with length at least $m_n^{c_n+\eps}$ is at most $\frac{m_n}{m_n^{c_n+\eps}}$. 

In  other words, there are at most $n^{c_n+\eps}$ segments in~\eqref{path}, where each segment consists of at most $M$ ``short'' edges and a single ``long'' edge. It follows that the length of the path~\eqref{path} is bounded by
$$
    (M+1) \cdot \frac{m_n}{m_n^{c_n+\eps}} \le \left(\frac{m_n^{c_n}}{\eps} +1\right)\frac{m_n}{m_n^{c_n+\eps}}
\le \left(\frac{2m_n^{c_n}}{\eps}\right)\frac{m_n}{m_n^{c_n+\eps}}=\frac{2}{\eps}m_n^{1-\eps}, 
$$  
and hence
$$
\codepth(S_n,G_n) \le \frac{2}{\eps}m_n^{1-\eps}.
$$
In particular, since $m_n = |V(G_n)|$, this establishes~\eqref{todo34} and finishes the proof.
\end{proof}

\paragraph{Circuit complexity} We finish this article by explaining some conjectural applications of hypershallow graph sequences to the theory of boolean circuits. As this is not of crucial importance for this article, we allow ourselves to be a little bit less precise for the sake of brevity. 

If $X$ is a set, then $2^X$ is the set of all functions from $X$ to $2=\{0,1\}$. This leads to the following notational clash: for $n\in \NN$, the symbol $2^n$ can either denote a number (and hence a set of numbers)  or the set of all functions from $\{0,1,\ldots, n-1\}$ to $\{0,1\}$. We believe that resolving this ambiguity will not cause any difficulty for the reader.

A convenient informal way of thinking about it is that if $k\in 2^n$ then $k$ is both a number smaller than $2^n$ and a function from $n=\{0,\ldots, n-1\}$ to $2=\{0,1\}$, and the translation between the two interpretations is that the binary expansion of a number $k$ smaller than $2^n$ can be thought of as a function from $\{0,\ldots, n-1\}$ to $\{0,1\}$.

A \emph{circuit} is a pair $\cal C = (G,\gate)$, where $G$ is a dag and  $\gate$ is a function which assigns to each vertex $v\in V(G)\setminus \IN(G)$ a function $\gate(v)\colon 2^{\In(v;G)}\to 2$. We will inherit the notation for $\cal C$ from the notation for $G$, thus e.g. we may write $\In(\cal C)$ for $\In(G)$.

For any $f\in 2^{\In(\cal C)}$ there exists exactly one function $F\in 2^{V(\cal C)}$ with the property that for every $v\in V(
\cal C)\setminus \IN(
\cal C)$ we have $\gate(v)(F|_{\In(v;\cal C)}) =F(v)$. In particular, we think of the restriction of $F$ to $\Out(G)$ as the output of the circuit $\cal C$ when $f$ is ``fed'' as the input.

Typically both $\IN(\cal C)$ and $\OUT(\cal C)$ have some labels, e.g.~both $\IN(\cal C)$ and $\OUT(\cal C)$ are labelled with elements of $\{0,\ldots, n-1\}$, in which case we may consider $\cal C$ to implement a function $2^n\to 2^n$.

By a simple counting argument, ``almost every'' sequence of functions 
$(f_n\colon 2^n\to 2^n)$ cannot be implemented by a sequence of bounded 
in-degree circuits $(\cal C_n)$ such that $|V(G_n)| = O(n)$. However, it is 
notoriously difficult to give ``explicit'' examples of sequences which cannot be computed by linear-sized circuits.

Following~\cite{afshani_et_al:LIPIcs:2019:10586}, let us state one precise question.

\begin{definition}\label{def-shift}
For $i\in \NN$ we let  $l(i) = \lceil\log(i)\rceil$, and we define $\shift_n\colon 2^{n\sqcup l(n)} \to 2^n$ as follows: 
if $f\in 2^n$ and $k\in 2^{l(n)}$, then for $j<n$ we let
$$
    \shift_n(f\sqcup k) (j) := f(j-k),
$$
where $j-k$ should be understood as an operation modulo $n$. In other words, $\shift_n(f\sqcup k)$ is equal to ``$f$ shifted by $k$''.
\end{definition}

\begin{Question}\label{q-shift}
Suppose that $(\cal C_n)$ is a bounded in-degree sequence of circuits which computes $\shift_n$. Is it the case that  $n = o(|V( G_n)|)$  ? 
\end{Question}

This innocent-looking question seems difficult to resolve (though there are some conditional results in~\cite{afshani_et_al:LIPIcs:2019:10586}). The authors of this article came up with the notion of hypershallow graph sequences motivated by the following strategy to attack this question: (1) ``Clearly'' if $(\cal C_n)$ is a hypershallow sequence which computes $\shift_n$, then $n = o(|V(\cal C_n)|)$, (2) Perhaps all graph sequences are hypershallow. 

The main result of this paper is that not all graph sequences are hypershallow (Theorem~\ref{thm-main-intro}). More annoyingly, the authors have not even been able to establish the first point of the above strategy.  As such, the following question is also open.

\begin{Question}\label{q-shift-easy}
Suppose that $(\cal C_n)$ is a bounded in-degree sequence of circuits which computes $\shift_n$ and which is hypershallow. Is it the case that $n = o (|V(\cal C_n|)$? 
\end{Question}

Let us finish this article by stating another question to which positive answer would imply a positive answer to Question~\ref{q-shift-easy}. We need to start with some  definitions.

An \emph{advice circuit} is a circuit $\cal C$ together with a partition 
of $\In(\cal C)$ into two disjoint subsets $\In_{std}(\cal C)$ and 
$\In_{adv}(\cal C)$. We think of such a circuit as receiving its input 
on the vertices in $\In_{std}(\cal C)$, together with some extra advice 
tailored specifically for a given input on the vertices in $\In_{adv}(\cal C)$. This is made precise in the 
following definition.

\begin{definition} Let $\cal C$ be an advice circuit. We say that $\cal C$ \emph{computes} $f\colon 2^{\In_{std}(\cal C)} \to 2^{\Out(\cal C)}$ if for every $s\in 2^{\In_{std}(\cal C)}$ there exists $t\in 2^{\In_{adv}(\cal C)}$ such that the output of $\cal C$ on $s\sqcup t$ is equal to $f(s)$. 
\end{definition}

An \emph{$\eps$-advice circuit} is an advice circuit $\cal C$ with $|\In_{adv}(\cal C)|\le \eps |\In(\cal C)|$.  With this we are ready to state the following question.

\begin{Question}\label{q-final}
Is it true that there exists $\eps>0$ such that the sequence $(\shift_n)$  cannot be computed by a sequence $(\cal C_n)$ of bounded in-degree $\eps$-advice circuits which have depth 1?
\end{Question}

It is not difficult to see that the positive answer to this question implies the positive answer to Question~\ref{q-shift-easy}.

\bibliographystyle{amsalpha}
\bibliography{biblio}
\end{document}